\documentclass[a4paper,oneside,11pt]{article} 

\usepackage[utf8]{inputenc} 
\usepackage[T1]{fontenc} 
\usepackage{textcomp} 
\usepackage[english]{babel} 
\usepackage{indentfirst}
\usepackage[autolanguage]{numprint} 
\usepackage{hyperref} 
\hypersetup{colorlinks=true,linkcolor=blue,citecolor=red,urlcolor=blue}
\usepackage{graphicx} 
\usepackage{verbatim} 
\usepackage{makeidx} 
\usepackage[refpage]{nomencl} 
\usepackage{enumitem} 
\usepackage{tikz} 
\usetikzlibrary{matrix,arrows,patterns}
\usepackage{multicol}
\usepackage{fullpage}
\usepackage{comment}
\usepackage{stmaryrd}
\usepackage{ulem}
\usepackage{mathrsfs}
\usepackage{mathtools}
\usepackage[hang,flushmargin]{footmisc}

\usepackage{amsmath,amssymb,amsthm,amsfonts} 
\usepackage[all]{xy} 

\numberwithin{equation}{section}

\newcounter{Main}

\theoremstyle{plain} 

\newtheorem{MainThm}[Main]{Theorem} 

\swapnumbers 
\theoremstyle{definition} 
\newtheorem{Def}{Definition}[section] 
\newtheorem{Def,Thm}[Def]{Definition and theorem} 
\newtheorem{Def,Prop}[Def]{Proposition-definition} 

\theoremstyle{plain} 
\newtheorem{Prop}[Def]{Proposition} 
\newtheorem{Lemma}[Def]{Lemma} 

\newtheorem{Theorem}[Def]{Theorem} 
\newtheorem{Corollary}[Def]{Corollary} 
\theoremstyle{remark} 
\newtheorem{Example}[Def]{Example} 
\newtheorem{Remark}[Def]{Remark} 

\newcommand{\x}{\mathrm{x}}
\newcommand{\A}{A_d}

\usepackage{xcolor}
\usepackage[all]{xy}
\newcommand{\red}[1]{{\color{red}#1}}

\newcommand{\w}{\varpi}

\newcommand\sbmattrix[4]{\textnormal{\scriptsize$\left(\begin{array}{cc}#1&#2\\#3&#4\end{array}\right)$\normalsize}}

\title{On some Diophantine and Ergodic properties of the Schneider map over function fields} 
\author{Matias Alvarado\footnote{Universidad de Talca, Instituto de Matemáticas, Talca, Chile. Email: \url{matias.alvarado@utalca.cl}}, \, Nicol\'as Ar\'evalo-Hurtado\footnote{Universidad Escuela Colombiana de Ingenier\'ia Julio Garavito, Bogot\'a, Colombia.
Email: \url{nicolas.arevalo-h@escuelaing.edu.co}}, \, and Claudio Bravo\footnote{Universidad de Talca, Instituto de Matemáticas, Talca, Chile. Email: \url{claudio.bravo@utalca.cl}}} 

\date{}

\begin{document} 

\maketitle

\begin{abstract}
We introduce and study continued fractions defined by Schneider-like maps over polynomial rings, where the maps are associated with a fixed polynomial of arbitrary degree.
In particular, we prove the existence and uniqueness of the continued fraction expansion for every element of the field of Laurent series. We then establish precise Diophantine approximation properties of the corresponding convergents.
We also study the dynamical aspects of the underlying map, proving that the Haar measure is invariant and ergodic. As an arithmetic consequence, we determine the asymptotic sum of the digits of the expansion for almost every element. Finally, we identify the set of elements that are worst approximable in this framework and compute its Hausdorff dimension, showing that it is a fractal set of positive dimension.
\end{abstract}

\section{Introduction}\label{section introduction}

The theory of continued fractions has a central role in number theory, serving as the primary tool for studying rational approximation of real and
$p$-adic numbers alike. In the classical archimedean setting, every irrational real number admits a unique expansion as an infinite simple continued fraction, whose convergents furnish the best rational approximations.
To be precise, each $\alpha \in \mathbb{R}$ is the limit of a unique sequence of the form 
\begin{equation}\label{eq0}
[a_0, \dots, a_N]:=a_0+\frac1{a_1+\frac{1}
{\ddots\frac1{a_{N-1}+\frac{1}{a_N}}}},
\end{equation}
where $N\in \mathbb{Z}_{\geq 0}$, $a_0 \in \mathbb{Z}$ and 
$a_1, \dots , a_N \in \mathbb{Z} \smallsetminus \lbrace 0 \rbrace$. 
This result can be interpreted in terms of the action of 
$\mathrm{SL}_2(\mathbb{Z})$ on the Poincar\'e upper half-plane endowed 
with the dual of the Farey tessellation (\cite{Ford,Series}).
Further connections between the development of real numbers in continued fractions and the ergodic theory of the Gauss map have been thoroughly developed and constitute a cornerstone of metric number theory; see for instance the monograph by Dajani and Kraaikamp~\cite{DK02}.

The analogue of continued fractions over function fields was introduced in earnest in the early works of Baum and Sweet \cite{BS76} and further developed by Mills and Robbins \cite{MR86}. In this setting, one replaces the ring of integers by a polynomial ring $\mathbb{F}_q[t]$ over a finite field and the field of real numbers by the field of formal Laurent series $\mathbb{F}_q(\!(t^{-1})\!)$. 
This construction extends to arbitrary ground fields. Specifically, Schmidt in \cite{Sc00} established the following result:

\begin{MainThm}\label{artincontinuedfractions}
Each $z\in F(\!(t^{-1})\!)$ 
can be developed in a unique continued fraction with polynomial 
coefficients, i.e., $z$ is the limit of a unique sequence of 
the form $[f_0, \dots, f_N]$ as defined above \eqref{eq0},
where $f_0 \in F[t]$ and 
$f_n \in F[t]\smallsetminus F$, for $n>0$.
Moreover, if we write 
$[f_0, f_1, \dots, f_n]=p_n/q_n$, 
with $p_n,q_n \in F[t]$ coprime, then:
$$ \nu\left(z-\frac{p_n}{q_n}\right)=\deg(f_{n+1})+2\deg(q_n).$$
In other words, we have: 
$$ \Bigg\vert z-\frac{p_r}{q_r} \Bigg\vert= \frac{ 1}{\vert f_{r+1} \vert\vert q_r \vert^2 }. $$
\end{MainThm}

The resulting theory is in many aspects richer and more rigid than its real counterpart, partly due to the non-Archimedean nature of the absolute value and partly due to the presence of positive characteristic. The foundational arithmetic aspects of this theory (including analogues of Dirichlet's theorem, Liouville's theorem, and Roth's theorem) were developed by Mahler \cite{Ma49}, Osgood \cite{Os75}, de Mathan \cite{dM92}, Lasjaunias \cite{La97, La00}, Lasjaunias–de Mathan \cite{LdM96}, Thakur \cite{Th99}, and Schmidt \cite{Sc00}, among others. The survey \cite{La00} by Lasjaunias offers an excellent overview of this classical theory.

In the non-Archimedean world, a different approach to continued fractions was pioneered by Schneider \cite{Sch70}, who defined a
$p$-adic continued fraction algorithm on $\mathbb{Z}_p$ by the map $x\mapsto p^{v(x)}/x - b
$ where $b$ is determined by reducing $p^{v(x)}/x$ modulo $p$, where $v$ denotes the $p$-adic valuation. This map, now called the Schneider map, was later studied systematically by Hirsh and Washington \cite{HirshWashington}, who proved that it preserves the Haar measure on $p\mathbb{Z}_p$ and is ergodic with respect to it, and who established an analogue of Khinchin's theorem for the resulting continued fraction expansion. See also \cite{alvaradoarevalo2} for generalizations of the analogue of Khinchin's theorem.  
The Schneider continued fractions can be directly extended to $\mathbb{F}_q(\!(t^{-1})\!)$ (see \cite{HN22} for instance) in the following way:


\begin{MainThm}\label{fraccion continua scneider polinomios} Let $z\in \mathbb{F}_q\llbracket t^{-1} \rrbracket$, and let $\pi=t^{-1}$. 
Then, there exist $b_0, b_1, b_2,... \in \mathbb{F}_q^\times$ 
and positive integers $a_1,a_2,...$, such that:
\begin{align}\label{SchneiderContinued}
      z= b_0+\cfrac{\pi^{a_1}}{b_{1}+\cfrac{\pi^{a_{2}}}{b_{2}+\cfrac{\pi^{a_{3}}}{\ddots}}}.
  \end{align}
\end{MainThm}

More recently, Arenas-Carmona and Bravo \cite{AB24} studied Diophantine approximation in local function fields via the theory of Bruhat–Tits trees, obtaining Dirichlet-type theorems with sharp constants and computing the measure of sets of elements that can be approximated better than predicted. Their approach uses the geometry of arithmetic quotients of the Bruhat–Tits tree as a replacement for continued fractions in contexts where these are not directly available (See \cite[Theorem 1.1 and Theorem 1.3]{AB24} for details).

The present work introduces and studies a new type of continued fraction expansion, generalizing the one introduced in Theorem \ref{fraccion continua scneider polinomios} to the function field context for arbitrary ground field $F$.
More explicitly, in this work we describe each element of the completion
$\hat{K}=F(\!(t^{-1})\!)$ of the fraction field $K=F(t)$, with respect to the valuation $\nu=-\deg$, as a continued fraction of the form \eqref{SchneiderContinued}, where $\pi$
is replaced by an arbitrary polynomial (in particular, one having
non-positive valuation). 
We refer to the resulting expansion as a
\emph{(generalized) Schneider continued fraction}. 
To be precise, fix a polynomial $P\in F[t]$  of degree $d\ge 1$, and define a (generalized) Schneider continued fraction as a continued fraction of the form:
\begin{equation}
[n_1,f_1, n_2,f_2, \cdots]:= \cfrac{P^{n_1}}{f_{1}+\cfrac{P^{n_{2}}}{f_{2}+{\ddots}}}, 
\end{equation}
where $f_i \in A_d$, for all $i\geq 1$, and $n_i \in \mathbb{Z}_{\leq 0}$ for all $i\geq 2$.
For each truncation of the form $[n_1,f_1, n_2,f_2, \cdots, n_N, f_N]$, we denote by $[n_1,f_1, n_2,f_2, \cdots, n_N, f_N]_{\mathrm{ev}}$ its value in $\hat{K}$.
In this way, denote by $[n_1,f_1, n_2,f_2, \cdots]_{\mathrm{ev}}$ the value in $\hat{K}$ of the limit $\lim_{N\to \infty}[n_1,f_1, n_2,f_2, \cdots,n_N,f_N]$, whenever it exists.

A Schneider continued fraction with integral part $f_0$ is a continued fraction of the form $f_0+[n_1,f_1, n_2,f_2, \cdots]$, where $f_0\in A$ and $n_1 \in \mathbb{Z}_{\leq 0}$.
Such continued fractions are called Schneider continued fraction with integral part in $A=F[t]$.
To construct the aforementioned continued fractions, we fix a
polynomial $P\in F[t]$ of degree $d\ge 1$ and introduce the
associated transformation $T_P\colon \hat{K}\to \hat{K}$, defined by $z \mapsto \frac{P^n}{z}-f,$
where the pair $(n,f)\in \mathbb{Z}\times A_d$ is uniquely determined by
$z$ via the Euclidean algorithm. Here
$A_d=\{Q\in F[t]:\deg Q<d\}$ denotes the set of residue
representatives modulo $P$. When $d=1$ the map $T_P$ looks similar to the classical Schneider setting. For $d\geq 2$, however, the situation is genuinely new: the 
elements $f_i$ are polynomials of degree up to $d-1$, 
the exponents $n_i$ may take arbitrary nonpositive 
integer values, and the resulting approximants 
$x_r(z)=p_r/q_r\in K$ satisfy a precise approximation formula. This is the content of our following theorem:

\begin{Theorem}\label{main theo 1}
For each $z \in \hat{K}^{\times}$, 
there exists a unique sequence 
$\lbrace (n_i,f_i) \rbrace_{i=1}^{N}$ of pairs with $f_i \in A_d$, $n_i \in \mathbb{Z}$ and $N\in \mathbb{Z}_{\geq 0}\cup \lbrace\infty\rbrace $ such that $\deg(f_i)-n_id>0$, for $i>0$ and
$z=[n_1, f_1, \dots]_{\mathrm{ev}}.$
Moreover, $N$ is finite exactly when $z \in K$.
In both cases, if we write 
$\x_r(z)=[n_1,f_1, \dots, n_r, f_r]_{\mathrm{ev}}$, for $r<N$, then
\begin{equation}\label{eq valuation}
\nu\big(z-\x_r(z)\big)=  2\sum_{i=1}^{r} \deg(f_i) + \deg(f_{r+1})-\sum_{i=1}^{r+1} n_i d.
\end{equation}
\end{Theorem}


At the end of \S \ref{sec existence of cf}, we prove that continued fractions defined by sequences $\lbrace (n_i,f_i) \rbrace_{i=1}^{N}$ with $f_i \in A_d$, $n_i \in \mathbb{Z}$ and $N\in \mathbb{Z}_{\geq 0}\cup \lbrace\infty\rbrace $ such that $\deg(f_i)-n_id>0$ are Cauchy. In particular, they are convergent in $\hat{K}$.

As the classical continued fractions in Theorem \ref{artincontinuedfractions}, Schneider continued fractions satisfy Diophantine approximation properties, as the following result shows.

\begin{Theorem}\label{main theo 2}
In the notation of Theorem \ref{main theo 1}, there exists $p_r, q_r \in A[\frac{1}{P}]$ such that $\x_r(z)=\frac{p_r}{q_r}$ and 
$$ \nu\left(z-\frac{p_r}{q_r} \right)=-2\nu(q_r)+\deg(f_{r+1})- \sum_{i=1}^{r+1} n_i d. $$
In other words:
$$\Bigg\vert z-\frac{p_r}{q_r} \Bigg\vert= \frac{\Big\vert P^{\sum_{i=1}^{r+1} n_i} \Big\vert}{\vert f_{r+1} \vert\vert q_r \vert^2 }.$$
\end{Theorem}

Beyond the arithmetic theory, in \S \ref{section dynamical aspecst} we study the dynamical and ergodic properties of $T_P$ in the particular case the ground field $F$ is the finite field $\mathbb{F}_q$ with $q$ elements. We add this hypothesis in order to have a Haar measure. In fact, we prove that the Haar measure on $\mathcal{M}$ is both invariant and ergodic under $T_P$, following the strategy of Hirsh and Washington \cite{HirshWashington} with careful adaptations to the function field setting. Via the ergodic theorem, this yields precise asymptotic frequency formulae for the exponents $\{n_i\}_{i=1}^{N}$ for Haar-almost every $z \in \hat{K}^{\times}$. 
We say that $f(r)\asymp_{q,d} g(r)$ if there exist constants $C_1$, and $C_2$ depending only on $q$ and $d$ such that \[C_1 f(r)\leq g(r)\leq C_2 f(r).\]
We prove the following result: 

\begin{Theorem}\label{main theo 3}
With respect to the Haar measure on $\hat{K}$, for almost every $z\in \mathcal{O}=\mathbb{F}_q\llbracket t^{-1}\rrbracket$, the asymptotic growth of the term $-\sum_{i=0}^rdn_i$ involved in Equation \eqref{eq valuation} is the following:
\[-\sum_{i=1}^r n_i\asymp_{q,d} r.\]
\end{Theorem}

A key feature of our construction, which has no analogue in the classical Schneider setting, is the set consisting of elements whose Schneider expansion uses only zero exponents. These are exactly the elements that are worst approximable in our framework. We compute the Hausdorff dimension of this set with respect to the $\nu$-adic topology. We prove this dimension is positive. In particular, we prove the following theorem.

\begin{Theorem}\label{main theo 4}
    
The set $\mathcal{Z}_0=\left\{ z\in \mathcal{M} \, | \, n_i=0 \text{ for all } i\geq 1 \right\}$
 has Hausdorff dimension 
 \[\mathrm{dim_{\mathrm{H}}}(\mathcal{Z}_0)=\dfrac{1}{2}\left(1+\dfrac{\log_q\left(q^{d-1}-1 \right)}{d-1} \right).\]
In particular, the set of elements  $z\in \mathcal{M}$ satisfying
\begin{align*}
\nu(z-\x_r(z))&=2\sum_{ i=1}^r\deg (f_i)+\deg(f_{r+1}),\, \text{ or equivalently }  \\
\left|z-p_r/q_r \right|&=\dfrac{1}{|f_{r+1}||q_r|^2},
\end{align*}
for any $r\geq 0$ is uncountable.
\end{Theorem}

\subsection*{Overview of the article} The article is organized as follows. \S \ref{sec existence of cf} introduces the generalized Schneider expansion via a recursive algorithm and proves the existence, and the approximation formula in Theorem \ref{main theo 1}). \S \ref{sec diophantine} establishes the Diophantine properties of the convergents, including the uniqueness of continued fractions and the precise approximation rate in terms of the denominators (Theorem \ref{main theo 2}). \S \ref{section dynamical aspecst} studies the dynamical and ergodic aspects of the map $T_P$: we prove invariance and ergodicity of the Haar measure (Theorems \ref{lemma Haar invariante} and \ref{medida ergodica}), and compute the Hausdorff dimension of the set $\mathcal{Z}_0$
of worst-approximable elements (Theorem \ref{main theo 4}).  Moreover we derive arithmetic consequences via the ergodic theorem, describing the asymptotic frequency of the exponents $n_i$ for Haar-almost every $z$. \S \ref{section examples} illustrates the theory with a family of explicit examples.

\section{Recursive definition of continued fractions}\label{sec existence of cf}

Let $F$ be a field and let $A=F[t]$ denote the polynomial ring in the variable $t$ with coefficients in $F$.
We write by $K$ for the fraction field of $A$.
The field $K$ carries a natural discrete valuation $\nu: K \to \mathbb{Z}\cup \lbrace \infty \rbrace$, defined by $\nu(f/g)=\deg(g)-\deg(f)$.
This valuation extends to a complete non-Archimedean field $\hat{K}=F((t^{-1}))$, which is the completion of $K$ with respect to $\nu.$
By abuse of notation, we continue to denote this valuation by $\nu: \hat{K} \to \mathbb{Z}\cup \lbrace \infty \rbrace$ for the induced valuation on $\hat{K}$. The ring of integers of $\hat K $ is $\mathcal{O}=F\llbracket t^{-1} \rrbracket=\{z\in \hat{K}|\,  \nu(z)\geq 0\}$, with unique maximal ideal $\mathcal{M}=t^{-1}F\llbracket t^{-1}\rrbracket=\{z\in \hat K| \nu(z)>0\}$, generated by the uniformizer $\pi=t^{-1}.$
In the sequel, we denote by $\mathcal{M}_{\neq0}$ the set $\mathcal{M} \cap \hat{K}^{\times}$.

Let $P=P(t)$ be a polynomial over $F$, and let $d=\deg(P)$.
We denote by $A_d$ the set of polynomials of degree stricly less than $d$, i.e., $\A=\lbrace Q=Q(t) \in A \vert \deg(Q) < d \rbrace$.

\begin{Lemma}\label{lemma first approx}
For each $z \in \hat{K}^{\times}$, there exists a 
unique pair $(f,n) \in \A\times \mathbb{Z}$ such that 
$\nu(zP^n-f) >0$ and $\nu(z)=nd-\deg(f)$.
Moreover $f\neq 0$.
\end{Lemma}

\begin{proof}
Let $z \in \hat{K}^{\times}$.
By Euclidean division, there exist $n\in \mathbb{Z}$ and $r \in \lbrace 1-d, \cdots, 0 \rbrace$ such that $\nu(z)=nd+r$, or equivalently, $\nu(zP^{n})=r$.
Write $zP^{n}= \sum_{i=r}^{\infty} a_i t^{-i}$, with 
$a_i \in F$ and $a_{-r} \neq 0$.
The polynomial $f=\sum_{i=r}^{0} a_i t^{-i}=\sum_{j=0}^{-r}a_{-j} t^{j} $ 
belongs to 
$\A$, satisfies $\nu(zP^{n}-f) >0$, and $\deg(f)=-r$, so that $f\neq 0$ and $\nu(z)=nd-\deg(f)$.
For uniqueness, suppose $n_1,n_2 \in \mathbb{Z}$ and $f_1, f_2 \in \A$ both 
satisfy $\nu(z)=n_id-\deg(f_i)$ and $\nu(zP^{n_i}-f_i)>0$ for $i=1,2$. 
Then $n_1=n_2$ and $\deg(f_1)=\deg(f_2)$, from which
$\nu(f_1-f_2)>0$. Since $f_1-f_2$ is a polynomial, this forces $\deg (f_1-f_2)<0$, hence $f_1=f_2.$
\end{proof}

\begin{Def}\label{def Alg of cf}
Write $\mathcal{M}_{\neq0}=\mathcal{M}\smallsetminus\{0\}$.
Given $z \in \hat{K}^\times $. We recursively define a sequence 
$\mathrm{App}(z)=\lbrace (f_i, n_i, z_i)\rbrace_{i=1}^{N} \subseteq \A \times \mathbb{Z}\times \hat{K}^\times$ by the following algorithm:
\begin{itemize}
    \item[\textbf{Step 1:}] Set $z_1=z$. 
    \item[\textbf{Step 2:}] Given $z_i \neq 0$, apply Lemma \ref{lemma first approx}  to $1/z_i$, to obtain the unique pair $(f_i,n_i)\in A_d\times \mathbb{Z}$ satisfying $\nu \left(\frac{1}{z_i}P^{n_i}-f_i \right)$ and $\nu(1/z_i)=n_i d-\deg f_i$. 
    \item[\textbf{Step 3:}] If 
    $\frac{1}{z_i}P^{n_i}-f_i\in\mathcal{M}_{\neq0}$, 
    set $z_{i+1}=\frac{1}{z_i}P^{n_i}-f_i$ and proceed to the next iteration.
    \item[\textbf{Step 4:}] If 
    $\frac{1}{z_i}P^{n_i}-f_i=0$, set $N=i$, terminate the algorithm, and declare that the continued fraction of $z$ terminates at step $i$.
\end{itemize}
\end{Def}

The following result is an immediate consequence of Definition \ref{def Alg of cf} and Lemma \ref{lemma first approx}.

\begin{Lemma}\label{p24}
The sequence $\mathrm{App}(z)$ satisfies the following properties:
 \begin{itemize}   
    \item[(i)] $f_i\neq 0$, for all $i\geq 1$, 
    \item[(ii)] $z_i \in \mathcal{M}_{\neq 0}$, for all $i\geq 2$,
    \item[(iii)] $\nu(z_i)=\deg(f_i)-n_id$.
    In particular $n_i \in \mathbb{Z}_{\leq 0}$ for all $i\geq 2$,
    \item[(iv)] if the algorithm terminates at some step N, then $z\in K$, and
    \item[(v)] if the algorithm does not terminate, the sequence $\lbrace (f_i, z_i) \rbrace_{i=1}^{\infty} \subseteq \A \times \mathcal{M}_{\neq0}$ is infinite.\qed
\end{itemize} 
\end{Lemma}

\begin{Remark}
Item $(v)$ in Lemma \ref{p24} holds precisely when $z\in \hat{K}\smallsetminus K$ according to Theorem \ref{main theo 1}, which we prove in \S \ref{sec diophantine}.
\end{Remark}

An infinite continued fraction (with numerators) is a formal expression $\w=[m_1,g_1,\dots]$ of the form:
$$ \w= \cfrac{P^{m_1}}{g_{1}+\cfrac{P^{m_{2}}}{g_{2}+{\ddots}}},$$

Given such an expression, we write 
$\w(r)=[m_1,g_1,\dots,m_r,g_r]$ for its $r$-th truncation, and $\x_r(\w)=\w(r)_{\mathrm{ev}} \in K$ for its evaluated value.

\begin{Def}
The continued fraction of $z$ is the expression
$\w_z=[n_1,f_1,n_2,f_2,\dots]$, where $\mathrm{App}(z)=\lbrace (f_i,n_i, z_i) \rbrace_{i=1}^{N}$. The elements $\x_r(z):=\x_r(\w_z)\in K$
are called the convergents of $z$.
\end{Def}

The algorithm above, together with the truncated continued fractions, gives rise to two auxiliary sequences $\{a_r\}$ and $\{b_r\}$, defined by the recurrences 
\begin{equation}\label{recurrence numerator}
b_1=-P^{n_1}, \quad b_2=-f_2P^{n_1}, \quad b_m=-P^{n_m}b_{m-2}-f_mb_{m-1} \text{ (for }m\geq 3),
\end{equation}
and
\begin{equation}\label{recurrence denominator}
a_1=f_1, \quad a_2= f_1f_2+P^{n_2}, \quad a_m=P^{n_m}a_{m-2}+f_ma_{m-1} \text{ (for }m\geq 3).
\end{equation}
In the preceding notation we have $\x_r=-b_r/a_r$. These sequences will be used in the proof of Theorem \ref{main theo 2} in \S \ref{sec diophantine}, and in the proof of ergodicity of the Haar measure with respect to $T_P$ (see Theorem \ref{medida ergodica} in \S \ref{section dynamical aspecst}). The following result is straightforward.

\begin{Lemma}\label{lema am y bm cruzados}
The following identity holds for all $m\geq 1$:
\[b_{m-1}a_{m}-b_ma_{m-1}=(-1)^{m-1}P^{n_1+\cdots +n_m}\]. \qed
\end{Lemma}

The remainder of this section is devoted to prove the following proposition, which establishes the existence of Schneider continued fractions approximating elements of $\hat{K}^{\times}$, and establishes the relation \eqref{eq valuation} concerning the rate of convergence of these convergents.
The uniqueness of continued fractions converting to a given element is proved in \S \ref{sec diophantine}. 

\begin{Prop}\label{prop approx cf}
For each $z\in \hat{K}\smallsetminus \lbrace 0 \rbrace$ whose continued fraction does not terminate
before the step $n+1$ is defined, we have the following identity:
\begin{equation}\label{eq approx cf}
\nu\big(z-\x_r(z)\big)=  2\sum_{i=1}^{r}\deg(f_i) +\deg(f_{r+1})-\sum_{i=1}^{r+1} n_i d.
\end{equation}
\end{Prop}

In order to prove Proposition \ref{prop approx cf} we introduce the following notations:

\begin{Def}\label{def rho}
Let $\lbrace (f_i,n_i) \rbrace_{i=1}^n$ be as in Definition \ref{def Alg of cf}. For each $i$ define $\sigma_i$ the M\"obius transformation given by $\sigma_i(x)= \frac{1}{x} P^{n_i}-f_i$, for $x \in \mathbb{P}^1(\hat{K})$, and set $\rho_n= \sigma_n \circ \cdots \circ \sigma_1$.
\end{Def}

The following result is a direct consequence of the definitions.

\begin{Lemma}\label{lemma cf in terms of rho}
Let $z\in \hat{K}\smallsetminus \lbrace 0 \rbrace$ be an element whose continued fraction does not terminate
before step $n+1$, and let $\mathrm{App}(z)=\lbrace (f_i, n_i, z_i)\rbrace_{i=1}^{N}$ be its 
associated sequence. Then, the following identities hold:
\begin{itemize}
\item[(1)] $\rho_n^{-1}(0)=\x_n(z)$, and
\item[(2)] $\rho_n(z)=z_{n+1}$.\qed
\end{itemize}
\end{Lemma}

\begin{Lemma}\label{lemma log derivate and degrees}
In the notation of Lemma \ref{lemma cf in terms of rho}, we have:
\begin{equation}\label{eq relation nu}
\nu \left(\frac{\rho_r(z)}{\rho_r'(z)}\right)= \sum_{i=1}^{r} (2\deg(f_i) - n_i d) + \deg(f_{r+1})- n_{r+1} d.  
\end{equation}
\end{Lemma}

\begin{proof}
We proceed by induction on $r$.
Firstly, for $r=1$, note that $\rho_1=\sigma_1$, so 
$(\rho_1)'(z)=-\frac{P^{n_1}}{z^2}$, and therefore $\frac{\rho_1(z)}{\rho_1'(z)}=-\frac{z^2}{P^{n_1}}(\frac{1}{z}P^{n_1}-f_1)=-\frac{z_1^{2}z_2}{P^{n_1}}$. Taking valuation yields 
\begin{equation}\label{eq aprox for r=0}
\nu \left(\frac{\rho_1(z)}{\rho_1'(z)}\right) = 2 \deg(f_1)-n_1d +\deg(f_2)-n_2d,
\end{equation}
as required.
Now assume that Equation \eqref{eq relation nu} holds for some $r \in \mathbb{Z}_{\geq 0}$.
Since $\rho_{r+1}=\sigma_{r+1} \circ \rho_r$, we have 
$\nu \left(\frac{\rho_{r+1}(z)}{\rho_{r+1}'(z)}\right)=
\nu \left(\frac{\sigma_{r+1}(\rho_{r}(z))}{\sigma_{r+1}'(\rho_r(z))
\cdot \rho'_r(z)}\right).$
Since $\sigma_{r+1}'(x)=-\frac{1}{x^2}P^{n_{r+1}}$, 
the valuation of $\frac{\rho_{r+1}(z)}{\rho_{r+1}'(z)}$ equals 
$ \nu\left(\frac{\rho_r(z)}{\rho_r'(z)} \right)+ \nu (\rho_r(z))+\nu(\rho_{r+1}(z)) + n_{r+1}d$.
Then, it follows from Lemma \ref{lemma cf in terms of rho} together with Lemma \ref{p24}(iii) that:
\begin{align*}
\nu \left(\frac{\rho_{r+1}(z)}{\rho_{r+1}'(z)}\right)&= \nu\left(\frac{\rho_r(z)}{\rho_r'(z)} \right)+ \nu (z_{r+1})+\nu(z_{r+2}) + n_{r+1}d,\\
&=\nu\left(\frac{\rho_r(z)}{\rho_r'(z)} \right)+ (\deg(f_{r+1})-n_{r+1}d)+(\deg(f_{r+2})-n_{r+2}d) + n_{r+1}d.
\end{align*}
Hence, the result follows from the inductive hypothesis.
\end{proof}

\begin{Lemma}\label{lemma valuation of cf}
Let $\w=[n_1, f_1, \dots]$ be either a finite or an infinite continued fraction
satisfying $\deg(f_i)-n_id>0$ for $i>0$. Assume $z=\w_{\mathrm{ev}}\in\hat K$ is
well-defined, i.e., either $\w$ is finite or 
the sequence $\x_n(\w)$ converges in $\hat{K}$.
Then, we have $\nu(z)= \deg(f_1)-n_1d$.
\end{Lemma}

\begin{proof}
Suppose first that $N \coloneq \text{length}(\w) < \infty$.
We prove the statement by induction on $N$.
For $N=1$, we have $z=\frac{P^{n_1}}{f_1}$, giving 
$$\nu(z)=\nu(P^{n_1})-\nu(f_1)=\deg(f_1)-n_1d.$$
Assume the statement holds for any expression with a
given length $N$.
Let $\w=[n_1, f_1, \dots, n_N, f_N]$ and
$\w'=[n_2, f_2, \cdots, n_N, f_{N}]$, with
$z=\w_{\mathrm{ev}}$ and 
$z'=\w'_{\mathrm{ev}}$. By the inductive hypothesis, $\nu(z')=\deg(f_2)-n_2d$. In particular $\nu(f_1)=\nu(f_1+z')$.
Therefore
$$\nu(z)=\nu\left(\frac{P^{n_1}}{f_1+z'}\right)=\nu(P^{n_1})-\nu(f_1)=\deg(f_1)-n_1d.$$
Finally, if $N=\text{length}(\w)=\infty$, then $z=\lim_{n\to \infty} \x_n(\w)$ and the conclusion follows by continuity of the valuation $\nu$. 
\end{proof}

\begin{proof}[Proof of Proposition \ref{prop approx cf}]
In the notation of Lemma \ref{lemma cf in terms of rho}, it suffices to establish that 
$$\nu\big(z-\x_r(z)\big)=
\nu \left(\frac{\rho_r(z)}{\rho_r'(z)}\right).$$
We proceed by induction on $r \in \mathbb{Z}_{\geq 0}$.
For $r=1$, since $\x_1(z)=\frac{P^{n_1}}{f_1}$,we have $z-\x_1(z)=\frac{z}{f_1}\cdot(f_1-\frac{1}{z}P^{n_1})$. Therefore
$$\nu(z-\x_1(z))=\nu(z)-\nu(f_1)+\nu(z_2)=2\deg(f_1)-n_1d + \deg(f_2) -n_2d.$$
Hence the result follows for $r=1$ from Equation \eqref{eq aprox for r=0}.

Now, assume that $\nu\big(z-\x_r(z)\big)=
\nu \left(\frac{\rho_r(z)}{\rho_r'(z)}\right)$, 
for some fixed $r\geq 0$, and
for all $z \in \hat{K}$ for which $z_{r+1}=\rho_r(z)$ is defined.
Let us write $\tau=\sigma_0(z)$.
The sequence $\mathrm{App}(\tau)$ is exactly 
$\lbrace (f_i,n_i, z_i)\rbrace_{i=2}^{N}$, a shift of 
$\mathrm{App}(z)$.
In particular, the Moebius transformation 
$\tilde{\rho}_r=\rho_{r+1} \circ \sigma_1^{-1}$
is precisely the $r$-th term of the sequence given by 
Definition \ref{def rho}, when $z$ is replaced by $\tau$.
Then, it follows from the inductive hypothesis that 
$\nu\big(\tau-\x_r(\tau)\big)=\nu 
\left(\frac{\Tilde{\rho}_r(\tau)}{\Tilde{\rho}_r'(\tau)}\right)$.
On one hand, it follows from Lemma \ref{lemma cf in terms of rho}(1) that 
$$\x_r(\tau)=\Tilde{\rho_r}^{-1}(0)=\sigma_1 \big(\rho_{r+1}^{-1}(0)\big)=\sigma_1\big(\x_{r+1}(z)\big),$$
so that
$$\nu\big(\tau-\x_r(\tau)\big)=\nu\Big(\sigma_1(z)-\sigma_1\big(\x_{r+1}(z)\big)\Big)=\nu(P^{n_1})+\nu(z-\x_{r+1})-\nu(z)-\nu(\x_{r+1}(z)).$$
Moreover, it follows from Lemma \ref{lemma valuation of cf} together with Lemma \ref{p24}\,(iii) that $\nu(\x_{r+1}(z))=\nu(z)$. Hence, we get:
\begin{equation}\label{eq 1 ind}
\nu \left(\frac{\Tilde{\rho}_r(\tau)}{\Tilde{\rho}_r'(\tau)}\right)=\nu\big(\tau-\x_{r+1}(\tau)\big)= \nu\big(z-\x_{r+1}(z)\big)+\nu(P^{n_1})-2\nu(z).   
\end{equation}
On the other hand, applying the chain rule 
in the denominator, we have 
\begin{equation}\label{eq 2 ind}
\nu \left(\frac{\Tilde{\rho}_r(\tau)}{\Tilde{\rho}_r'(\tau)}\right)=\nu \left(\frac{\rho_{r+1}(z)}{\rho_{r+1}'(z)} \right)+\nu \big( \sigma_1'(z)\big). 
\end{equation}
It is straightforward that 
$\nu \big( \sigma_1'(z)\big)=\nu(P^{n_1})-2\nu(z)$.
Thus, it follows from Equations \eqref{eq 1 ind} and 
\eqref{eq 2 ind} that 
$\nu\big(z-\x_{r+1}(z)\big)=\nu 
\left(\frac{\rho_{r+1}(z)}{\rho_{r+1}'(z)}\right)$, 
whence the result follows.
\end{proof}

\begin{Remark}
Let $\{(n_i,f_i)_{i=1}^m\}$ be a sequence of elements in $A_d\times \mathbb{Z}$ such that $\deg f_i-n_i d>0$, which defines a truncated continued fraction. In order to ensure convergence of the continued fraction, it suffices to verify that the sequence of convergents $\x_m$ is Cauchy. Recall that $\x_m=-b_m/a_m$, where $b_m$ and $a_m$ are defined in 
Equations \eqref{recurrence numerator} and \eqref{recurrence denominator}. By Lemma \ref{lema am y bm cruzados} we have
\[\x_{m+1}-\x_{m}=\dfrac{(-1)^{m}P^{n_1+\cdots+n_m}}{a_ma_{m+1}}\]
Recalling that the valuation of denominator is $-\deg f_{m+1}-2\sum_{i=1}^{m} \deg f_i,$ and $\deg f_i-n_i d>0$, we obtain,
\[\nu(\x_{m+1}-\x_{m})=\sum_{i=1}^m(-d n_i+\deg f_i)+\sum_{i=1}^{m+1} \deg f_i\to \infty \text{ as } m\to \infty.\]
We conclude that, under the hypothesis $\deg f_i-n_i d>0$, the sequence converges. 
\end{Remark}

\section{Diophantine properties}\label{sec diophantine}

In this section we establish the uniqueness of the Schneider continued fraction expansion (See Proposition \ref{prop cf are uniq} below).
Using that, together with Propoposition \ref{prop approx cf}, we prove Theorems \ref{main theo 1} and \ref{main theo 2}.

\begin{Lemma}\label{lemma equal cf}
Let $z,z'\in \hat{K}^\times$ be two elements with $\mathrm{App}(z)=\{(f_i, n_i, z_i)\}_{i=1}^{N}$,
and $\mathrm{App}(z')=\{(g_i,m_i, z'_i)\}_{i=1}^{M}$. Fix $r\geq 0$ and assume that $z_r$ is defined, that is, the continued fraction
of $z$ does not terminate before the $r$-th step in the sense of \textbf{Step 3 \& 4} of
Definition \ref{def Alg of cf}.
Suppose furthermore that
$$\nu(z-z') > 2\sum_{j=1}^{r-1}\deg(f_j) +\deg(f_{r}) - \sum_{j=1}^{r}n_j d.$$
Then $z'_r$ is defined, $n_i=m_i$ for all $i \leq r$ and $f_i=g_i$ for all $i \leq r-1$.
\end{Lemma}

\begin{proof}
We proceed by induction on $j\leq r$.
For $j=1$, we have $\nu(z-z')>\deg(f_1)-n_1d$, whence
$$ \deg(g_1)- m_1d=\nu(z')=\nu\big(z+(z'-z)\big)=\nu(z)=\deg(f_1)-n_1d.$$
Hence $n_1=m_1$. Note that $z_1=z$ and $z_1'=z'$ are always defined.

Now, for $j=2$, we have $n_1=m_1$ arguing as above.
Observe that:
$$\nu\left(P^{n_1}\left(\frac{1}{z}-\frac{1}{z'}\right)\right)= n_1d-2\deg(f_1) +\nu(z-z')>\deg(f_2)-n_2d>0.$$
Therefore
$\nu\left(P^{n_1}\frac{1}{z'}-f_1\right)= \nu\big((P^{n_1}\frac{1}{z}-f_1)+P^{n_1}(\frac{1}{z'}-\frac{1}{z})\big)>0$.
Then, it follows from the uniqueness of polynomials in Lemma \ref{lemma first approx} that $f_1=g_1 \in \A$.
Therefore, we can prove as in the case $j=1$ that $n_2=m_2$.
Observe that $z_2-z_2'=P^{n_1} \left( \frac{z'-z}{zz´} \right) $, whence $\nu(z_2-z_2')>\deg(f_2)-n_2d.$
Moreover, note that $z_2'$ is undefined exactly when $P^{n_1}/z'-f_1=0$.
In that case, $\nu(z_2)>\deg(f_2)-n_2d$, which contradicts Lemma \ref{p24}\,(iii).
Therefore, $z_2'$ is defined.

Now, assume that $z_i'$ is defined and that $n_i=m_i$, $f_{i-1}=g_{i-1}$ for all $i\leq j \leq r$.
We will show, by hard induction, that $z_{j+1}'$ is defined, $n_{j+1}=m_{j+1}$ and that $f_j=g_j$.
Indeed, by direct induction we prove that
$\nu(z_{j+1}-z'_{j+1}) >\sum_{i=j+1}^{r-1} (2\deg(f_i)-n_id) + \deg(f_{r})-n_{r}d$,
arguing as in the previous paragraph.
In particular $\nu(z_{j+1}-z'_{j+1})>\deg(f_{j+1})-n_{j+1}d$, whence $n_{j+1}=m_{j+1}$ by the same argument as given above.
Note that $z_{j+1}'$ is defined, since $\nu(z_{j+1})>\deg(f_{j+1})-n_{j+1}d$ in any other case. This contradicts Lemma  \ref{p24}\,(iii).
Finally, since 
$$\nu\left(P^{n_j}\left(\frac{1}{z_j}-\frac{1}{z'_j}\right)\right)= n_jd-2\deg(f_j) +\nu(z_j-z_j')>\deg(f_{j+1})-n_{j+1}d>0,$$
we have $\nu\left(P^{n_j}\frac{1}{z_j'}-f_j\right)= \nu\big((P^{n_j}\frac{1}{z_j}-f_j)+P^{n_j}(\frac{1}{z'_j}-\frac{1}{z_j})\big)>0$.
Therefore $f_j=g_j \in \A$, by uniqueness in Lemma \ref{lemma first approx}.
\end{proof}

We define the length of a continued fraction $\w$, denoted by $\text{length}(\w)$, to be the largest integer $N$ such that the algorithm does not terminate before step $N$. If no such largest integer exists, we set $\text{length}(\w)=\infty$.

\begin{Prop}\label{prop cf are uniq}
Suppose $z=\w_{\mathrm{ev}}=\w'_{\mathrm{ev}}$, for two continued fraction expressions $\w=[f_1, n_1, \dots]$ and $\w'=[g_1, m_1, \dots]$.
Then $\text{length}(\w)=\text{length}(\w')$, $n_i=m_i$ and $f_i=g_i$, for all $i\geq 0$.
\end{Prop}

\begin{proof}
Without loss of generality, assume $N=\text{length}(\w)\leq M=\text{length}(\w')$.
Moreover, assume that $N<\infty$. 
Then $z_{N+1}$ is undefined, so Lemma \ref{lemma equal cf} implies that $z'_{M+1}$ is likewise undefined giving $M=N$. An analogous argument handles the case $M<\infty$.
Applying Lemma \ref{lemma equal cf}, we obtain $n_i=m_i$ for all $i\leq N$ and $f_i=g_i$, for all $i\leq N-1$.
Since $\text{length}(\w)=\text{length}(\w')=N$, the element $z_N$ can be expressed both as $\rho_{N-1}(z)=P^{n_N}/f_N$ and as $\rho'_{N-1}(z)=P^{n_N}/f_N'$, where $\rho_k'=\sigma_k'\circ \cdots \circ \sigma_1'$ with $\sigma_i'(x)=\frac{1}{x}P^{m_i}-g_i$.
Thus, we conclude from the previous equalities that $f_N=f_N'$ as desired.
Now, if $l(\w)=\infty$, Lemma \ref{lemma equal cf} directly yields $n_i=m_i$ and $f_i=g_i$, for all $i\geq 0$, whence the result follows.
\end{proof}

\begin{proof}[Proof of Theorem \ref{main theo 1}]
The first statement immediately follows from Proposition \ref{prop approx cf} together with Proposition \ref{prop cf are uniq}.
Moreover equation \ref{eq valuation} follows also from Proposition \ref{prop approx cf}.
If the algorithm terminate at step $N$, then $z=\x_r(z) \in K$.
It remains to show that the algorithm terminates in finitely many steps if and only if $z\in K$.
Now, suppose $z=\frac{a}{b}$ with $a,b \in A$ coprime. There exists $m_1 \in \mathbb{Z}$ and $\kappa \in \lbrace 0, \cdots, d-1 \rbrace$ such that $\deg(b)-\deg(a)=-m_1d+\kappa_1$, that is, $\deg(P^{m_1} b)=\deg(a)+\kappa$. 
By the Euclidean algorithm applied to $P^{m_1}b$ and $a$, there exist $g_1 \in A$ with $\deg(g_1)=\kappa$ (so that, $g_1 \in A_d$) and $r_2 \in A$ with $\deg r_2<\deg a$ such that $P^{m_1}b= g_1 a+ r_2$. This gives

$$ z=\frac{a}{b}=\frac{P^{m_1}}{ g_1+ r_2/a}.$$
Applying the same procedure to $r_2/a$ yields $m_2\in \mathbb{Z}$ and $g_2 \in A_d$, and one verifies that $\deg(g_2)-m_2d=\deg(a)-\deg(r_2)>0$.
Iterating, we obtain sequences
$m_i\in \mathbb{Z}$, $g_i\in A_d$ and $r_{i+1} \in A$ satisfying:
\begin{itemize}
\item $\deg(r_{i-1})-\deg(r_i)=-m_id+\kappa_i$, for some $\kappa_i \in \lbrace 0, \cdots, d-1\rbrace$,
\item $P^{m_i} r_{i-1}= g_i r_i + r_{i+1}$, with $\deg(r_{i+1})< \deg(r_i)$,
\item $\deg(g_i)=\kappa_i$.
\end{itemize}
Since
$$\deg(a)>\deg(r_1) >\cdots > \deg(r_i)\geq 0,$$
there exists an index $I\in \mathbb{Z}$ such that $r_I=0$, and therefore
$$x=\cfrac{P^{m_{1}}}{g_1+\cfrac{p^{m_{2}}}{{\ddots + \cfrac{P^{m_N}}{g_I}}}}.$$

As $0<-m_id+\deg(g_i)$ for all $i$, Proposition \ref{prop cf are uniq} identifies this as the unique continued fraction of $z$, which is finite. Hence the algorithm terminates after finitely many steps, as claimed.
\end{proof}

\begin{Example}\label{ex affine cf}
We can define Schneider continued fractions with integral part in $A$ approximating to elements in $\hat{K}^{\times}$ as follows: Given $z \in \hat{K}^{\times}$, there exists a unique polynomial $f_0 \in A= F[t]$ such that $z-f_0 \in \mathcal{M}$ according to \cite[\S 2, Pag. 68]{Paulin}.
Hence, we apply Theorem \ref{main theo 1} to $z'=z-f_0$. In particular, we have $n_1\in \mathbb{Z}_{\leq 0}$ since $\nu(z')\geq 0$.
The rate of convergence of such continued fraction is given by Equation \eqref{eq valuation} since the approximants of $z$ have the form $\x_r(z)= \x_r(z')+f_0$.
This exhibits a direct method to transform the generalized Schneider continued fraction to Schneider continued fractions with integral part in $A$, which are analogous to that introduced in Theorem~\ref{fraccion continua scneider polinomios}.
\end{Example}

\begin{proof}[Proof of Theorem \ref{main theo 2}]
By Theorem \ref{main theo 1} we have $\nu\big(z-\x_r(z)\big)=  2\sum_{i=1}^{r}\deg(f_i) +\deg(f_{r+1})-\sum_{i=1}^{r+1} n_i d.$
Recall from Lemma \ref{lemma cf in terms of rho}\,(1) that $\rho_{r}^{-1}(0)=\x_n(z)$.
With $\rho_r=\sigma_r \circ \cdots \circ \sigma_1$ and $\sigma_i(x)=\frac{P^{n_i}}{x}-f_i$, consider the lifting $M_r$ of $\rho_r$ given by 
$$M_r= \sbmattrix{-f_r}{P^{n_r}}{1}{0} \cdots \sbmattrix{-f_1}{P^{n_1}}{1}{0}=:\sbmattrix{a_r}{b_r}{c_r}{d_r},$$
where $a_r,b_r,c_r,d_r \in A[\frac{1}{P}]$.
In terms of this lifting, $\rho_r(x)=\frac{a_rx+b_r}{c_rx+d_r}$ for $x \in \mathbb{P}^1(\hat{K})$, and therefore $\rho_r^{-1}(x)=\frac{d_rx-b_r}{-c_rx+a_r}$, from which 
$\x_n(z)=\rho_r^{-1}(0)=\frac{-b_r}{a_r}.$
We claim that
\begin{equation}\label{eq valuation of ar}
\nu(a_r)=-\sum_{i=1}^{r}\deg(f_i).
\end{equation}
For $r=1$, we have $a_1=f_1$, and the identity is clear. For $r=2$ , we have $a_2=f_1f_2+P^{n_2}$.
On one hand, if $n_2<0$, then $\nu(a_2)=\nu(f_1f_2)=-\deg(f_2)-\deg(f_1)$.
On the other hand, if $n_2=0$, then $\deg(f_2)>0$ by Lemma \ref{p24}\,(ii)-(iii), so $\nu(a_2)=\nu(f_1f_2)=-\deg(f_2)-\deg(f_1)$, giving the same conclusion.
Now, assume that Equation \eqref{eq valuation of ar} holds for all integers strictly less than $r$.
Note that $a_r$ satisfies the recurrence $a_r=P^{n_r}a_{r-2}+f_r a_{r-1}$ (see Equation \eqref{recurrence denominator}). It also holds $\nu(a_{r-1}) \leq \nu(a_{r-2})$ since $\deg(f_{r-1})\geq 0$. If $n_r<0$, then $\nu(f_ra_{r-1}) < \nu(P^{n_r}a_{r-2})$, since $\deg(f_r)< d=\deg(P)$, whence $\nu(a_r)=\nu(f_ra_{r-1})=-\sum_{i=1}^{r}\deg(f_i)$.
On the other hand, if $n_r=0$, then $\deg(f_r)>0$, according to Lemma \ref{p24}\,(ii)-(iii), whence $\nu(f_ra_{r-1}) < \nu(a_{r-2})$. In particular $\nu(a_r)=\nu(f_ra_{r-1})=-\sum_{i=1}^{r}\deg(f_i)$, as claimed.
Therefore, Equation \eqref{eq valuation of ar} holds. 
Hence, the result follows by taking $p_r=-b_r$ and $q_r=a_r$.
\end{proof}


\section{Dynamical aspects behind the continued fractions}\label{section dynamical aspecst}

Questions concerning the quality of approximation of elements of $\hat{K}$ arise naturally in Diophantine approximation theory. In light of Theorem \ref{main theo 1} the approximation quality is governed by the quantities $-\sum d n_i$ and $\deg f_i$ for all $i$. Henceforth we assume the ground field $F$ the finite field $\mathbb{F}_q$ with $q$ elements. In this section, we will focus on studying $-\sum dn_i$. Using ergodic methods, we establish several properties of this quantity. In particular, we show that its asymptotic behavior is the same for Haar-almost every element of $\hat{K}$, and that there exist infinitely many elements for which $n_i=0$ for every $i$. 
We begin with a lemma that will be used repeatedly throughout this section.

\begin{Lemma}\label{lemma deg f positivo}
    Let $z\in \hat K$. Then $n_1=0$ if and only if $\nu(z)\in \left\{0,...d-1 \right\}$. Furthermore, if $z\in \mathcal{M}$, then $n_1=0$ implies $\deg f >0.$
\end{Lemma}

\begin{proof}
    If $n_1=0$, then $\nu(1/z-f)>0$, which is equivalent to
    \[\dfrac{1}{z}=f+\sum_{k\geq 1} a_k\left( 1/t\right)^k, \text{ where }a_k\in \mathbb{F}_q.\]
    It follows that $\nu(z)=-\nu(f)=\deg f\in \{0,...,d-1\}.$
Conversely, if $\nu(z)\in \{0,...,d-1\}$, then $n_1=0$ directly from the definition. Finally, if $z\in \mathcal{M}$, then $\nu(z)>0,$ so $\deg f=\nu(z)>0.$
\end{proof}

We now introduce the map whose dynamics encode the continued fraction expansion, called generalized Schneider map.

Before proceeding, we briefly recall the classical Schneider map and contrast it with the map introduced here. Let $\mathscr{O}$ be the valuation ring of a local field with valuation $v$ and uniformizer $\pi$. The classical Schneider map
 $T\colon \pi\mathscr{O}\to \pi\mathscr{O}$ is defined by $x\to \pi^{v(x)}/x-b$, where $b\in \mathscr{O}^\times$ is the unique unit satisfying $\pi^{v(x)}/x\equiv b \text{ mod}\, \pi.$

\begin{Def} We define the generalized Schneider map as the map $T_P\colon \hat{K} \to \hat K$ defined by $z\mapsto\frac{P^n}{z}-f$, where the pair $(n,f)$ is as in Lemma \ref{lemma first approx}.
\end{Def}

 Since $\nu(P^n/z-f)>0$ by construction, the image of $T_P$ is contained in $\mathcal{M}$, so $T_P$ induces a well-defined dynamical system on $\mathcal{M}$.

The essential difference between the two maps is that, in our setting, the role of the uniformizing parameter is played by a polynomial $P$ that belongs outside the valuation ring and may have arbitrary negative valuation. As a consequence, the corresponding element $f$ does not need to be a unit of $\mathscr{O}$.

Repeated application of $T_P$ to any $z\in \hat{K}$ recovers the continued fraction expansion of $z$ via the identity

\begin{align}\label{SchneiderContinued}
      z= \cfrac{P^{n_{1}}}{f_{0}+\cfrac{p^{n_{2}}}{f_{1}+\cfrac{P^{n_{3}}}{\ddots + \cfrac{P^{n_{m-1}}}{f_{m-1}+T_{P}^m}}}}.
  \end{align}

  \noindent In this way, the digits $n_i$ and $f_i$ introduced in \S \ref{sec existence of cf} are recovered inductively as $n_i(z)=n_1\left(T_P^{i-1}(z)\right)$, and $f_i(z)=f_0(T_P^{i-1}(z))$.

Let $\mu$ denote the Haar measure on $\hat{K}$, normalized so that $\mu\left(\mathcal{M}\right)=1$; see \cite{Foland} for background on Haar measure.  In the sequel, we prove that $\mu$ is both invariant and ergodic with respect to $T_P$, following the strategy of Hirsh and Washington \cite[Lemma 1]{HirshWashington}. Although the structure of the arguments is analogous to that of \cite{HirshWashington}, the essential difference between $T_P$ and the classical map requires careful adaptation at several points.

\begin{Theorem}\label{lemma Haar invariante}
    The Haar measure $\mu$ is invariant under $T_P$, that is, $\mu(T_P^{-1}(X))=\mu(X)$ for every measurable set $X.$
\end{Theorem}
\begin{proof}
Let $a\in \mathcal{O}$ and let $n$ be a positive integer. It suffices to verify the identity for sets of the form $X=\pi a+\pi^n \mathcal{O}$. Indeed, a direct computation shows that for any $(m,f)\in \mathbb{N}\times A_d\, \cup\, \{0\}\times A_d^{>0}$
\[T_P\left( \dfrac{P^{-m}}{f+\pi a+\pi^n \mathcal{O}}\right)=\pi a +\pi^n \mathcal{O}.\]

\noindent and hence, by Lemma \ref{lemma deg f positivo} 
\[T_P^{-1}\left (\pi a+\pi^n \mathcal{O} \right) = \bigcup_{m=1}^\infty \bigcup_{f\in A_d} \dfrac{P^{-m}}{f+\pi a+\pi^n \mathcal{O}} \cup \bigcup_{f\in A_d^{>0}}\dfrac{1}{f+\pi a +\pi^n \mathcal{O}}.\]

\noindent Since
\[\dfrac{P^{-m}}{f+\pi a+\pi^n \mathcal{O}}=\dfrac{P^{-m}}{f+\pi a}+f^{-2}P^{-m}\pi^n\mathcal{O},\]

\noindent we obtain 

\begin{align*}
    \mu \left(T_P^{-1}\right)
    &=\sum_{f\in A_d}\sum_{m=1}^\infty \mu\left( f^{-2}P^{-m}\pi^n\mathcal{O} \right)+\sum_{f\in A_d^{>0}}\mu \left(f^{-2}\pi^n \mathcal{O} \right),\\
    &=q^{1-n}\left[ \left(\sum_{f\in A_d}q^{-2\deg f} \right)\left(\sum_{m=1}^\infty q^{-md} \right)+\sum_{f\in A_d^{>0}}q^{-2\deg f}\right].
\end{align*}

\noindent Associating terms by degree and counting polynomials of each degree, the previous equality becomes

\[q^{1-n}\left[\left(\sum_{r=0}^{d-1}q^r(q-1)q^{-2r}\right)\left(\sum_{m=1}^\infty \left(1/q^d\right)^m \right)+\sum_{r=1}^{d-1}q^r(q-1)q^{-2r} \right],\]

\noindent which equals $\mu(\pi a+\pi^n \mathcal{O})$. Therefore $\mu$ is invariant under $T_P$.
\end{proof}

\begin{Theorem}\label{medida ergodica}
    The Haar measure $\mu$ is ergodic under $T_P.$
\end{Theorem}
\begin{proof}
    Let $X\subset\mathcal{M}$ be a measurable set satisfying $T_P^{-1}(X)=X$. Since the set of elements of $\hat{K}$ with finite continued fraction expansion is precisely $K$, which is countable, we may restrict our attention to elements $x\in X$ with infinite expansion without affecting any measure theoretic conclusion.

    Fix such an $x\in X$, and consider the sequences $a_m$ and $b_m$ before defined. Recall they satisfies 
   \[-b_m/a_m= \cfrac{P^{n_{1}}}{f_{1}+\cfrac{p^{n_{2}}}{f_{2}+\cfrac{P^{n_{3}}}{\ddots + \cfrac{P^{n_{m}}}{f_{m}}}}}.\]

\noindent Define the map
    \[g(z)= \cfrac{P^{n_{1}}}{f_{1}+\cfrac{p^{n_{2}}}{f_{2}+\cfrac{P^{n_{3}}}{\ddots + \cfrac{P^{n_{m}}}{f_{m}+z}}}},\]

    \noindent which satisfies
    \[g(z)=\dfrac{-zb_{m-1}-b_m}{za_{m-1}+a_m}.\]

   \noindent  By Lemma \ref{lema am y bm cruzados}, this yields
    \[g(z)+\dfrac{b_m}{a_m}=\dfrac{(-1)^mzP^{n_1+\cdots+n_m}}{(za_{m-1}+a_m)a_m}.\]

    \noindent We set $N=\sum_{i=1}^m n_i.$ For $z\in \pi \mathcal{M}$, it follows that 
    \[g(z)\equiv -b_m/a_m \text{ mod }\pi^{-dN+1+2\sum \deg f_i}\mathcal{O}.\]

    \noindent Let $\chi$ denote the characteristic function of $X$. Since $T_P^{-1}(X)=X$ and $T_P^{m}\circ g(z)=z$ for all $z\in \mathcal{M}$, we have
    \[\chi(g(z))=\chi(T_P^mg(z))=\chi(z).\]

    \noindent Then, if $y=g(w)$, the condition $y\in x+\pi^{-d(n_0+\cdots +n_m)-1+2\sum deg f_i}\mathcal{O}$ is equivalent to the condition $w\in \mathcal{M}.$ Arguing as in \cite[Lemma 3 p.401]{HirshWashington}, one verifies that  $d\mu(g(w))=q^{d(n_0+\cdots +n_m)-2\sum \deg f_i}d\mu(w),$

    and therefore

    \[\mu\left(X\cap \left(x+\pi^{-d\sum n_i-1+2\sum \deg f_i}\mathcal{O}\right)\right)=q^{d\sum n_i-2\sum \deg f_i}\int_{w\in \mathcal{M}}\chi(w)d\mu(w)=q^{d\sum n_i-2\sum \deg f_i} \mu(X).\]

    \noindent The conclusion that $\mu(X)\in \left\{0,1 \right\}$ now follows by a density argument entirely analogous to \cite[Lemma 2 p. 402]{HirshWashington}.
\end{proof}

Now, we focus on computing the Hausdorff dimension of the set 
\[\mathcal{Z}_0\coloneq \{z\in \mathcal{M} \, | \, n_i(z)=0 \text{ for all } i\geq 1 \},\] 
which, consists precisely of the elements $z$ whose approximant $\x_r(z)=p_r/q_r$ (with $p_r$ and $q_r \in A[1/P]$) satisfies 
\[\nu(z-\x_r(z))=2\sum_{i=1}^r \deg f_i+\deg f_{r+1}.\]
Note that these elements are the worst approximable in our framework.

\begin{Remark}

The condition $n_i(z)=0$ for all $i\geq 1$ can only be satisfied when $\deg P >1$ (see Lemma \ref{lemma deg f positivo}). Indeed, when $d=1$ it follows from Lemma \ref{p24} (iii) that $n_i<0$ for all $i>0$, so $\mathcal{Z}_0=\emptyset$ in that case.
\end{Remark}

\noindent We briefly recall the relevant definitions. For further details, see \cite[Section 2.2]{fa}. For a subset $Q\subset \mathcal{M}$, $s\geq 0$ and $\delta>0$, define 
\begin{align*}
    \mathcal{H}^{s}_{\delta}(Q)\coloneq \inf\left\{\sum_{i}|U_{i}|^{s}:\begin{array}{c}
         \text{ $\{U_{i}\}$ is a countable cover of $Q$ such that for every $i$}  \\
         \text{ $|U_{i}|<\delta$.}
    \end{array} \right\}, 
\end{align*} where $|U_{i}|$ denotes the diameter of $U_{i}$ with respect to the ultrametric induced by the valuation $\nu(\cdot)$ normalized in a way that $|\pi|=q^{-1}$. Setting $\mathcal{H}^{s}(Q)=\lim_{\delta\rightarrow 0}\mathcal{H}^{s}_{\delta}(Q)$, the Hausdorff dimension of $Q$ is
\begin{align}\label{HAUSDORFF}
   \mathrm{dim}_{\mathrm{H}}(Q)\coloneq \inf\{s:\mathcal{H}^{s}(Q)=0\}.
\end{align}

\begin{Def}
Given $n\in \mathbb{Z}$, and $f_0,f_1,...,f_r\in A_d$, the cylinder set associated to the data $(n;f_0,...,f_r)$ is
\[I_{(n;f_0,...,f_r)}\coloneq \{z\in \hat K\, | \, n_i(z)=n, \text{ and }f_i(z)=f_i \text{ for all } i\leq r\}.\]
\end{Def}

\noindent For each integer $r\geq 0$, the collection $\mathcal{U}_r=\left\{I_{(0;f_1,...,f_r)} |  f_0,...,f_r \in A_d^{>0} \right\} $ forms a cover of $\mathcal{Z}_0$. Here, $A_d^{>0}$ denotes the set of elements in $A_d$ of positive degree. The ultrametric property implies that cylinder sets are balls. The next result is straightforward

\begin{Lemma}
    For any $x\in I_{(0;f_0,...,f_r)}$, we have
    \[I_{(0;f_0,...,f_r)}=x+\pi^{1+2\sum_{j=0}^r\deg f_j}\mathcal{O}.\]
    In particular, the diameter of this cylinder is $q^{-1-2\sum_{j=0}^r \deg f_j}$. \qed 
\end{Lemma}

We are now in a position to prove Theorem \ref{main theo 4}.

\begin{proof}[Proof of Theorem \ref{main theo 4}]

We compute $\mathcal{H}_\delta^s(\mathcal{Z}_0)$ explicitly by taking the covering at which the infimum is attained. The ultrametric property guarantees that the covering by balls of a fixed diameter is optimal for this purpose. We work with the sequence of $\delta=\delta_n=q^{-(1+2(d-1)n)}$, which tends to $0$ as $n\to \infty$, and compute $\mathcal{H}_{\delta_n}^s(\mathcal{Z}_0)$ by refining each cylinder in $\mathcal{U}_n$ into balls of diameter $\delta_n.$ Indeed, a direct computation gives:

\begin{align*}
\mathcal{H}^s_{\delta_n}(\mathcal{Z}_0)&=\sum_{I_{(0;f_1,...,f_n)}}\sum_{\substack{U_i\subset I_{(0,f_1,...,f_n)}\\|U_i|=q^{-(1+2(d-1)n)}}} |U_i|^s,  \\
&=\sum_{I_{(0;f_1,...,f_n)}} q^{(-2(d-1)n-1)s}\cdot q^{2\sum \deg f_{i}-2(d-1)n},\\
&=q^{-s(2(d-1))n-1}q^{2(d-1)n}\sum_{I_(0;f_1,...,f_n)}q^{-2\sum_{i\leq n}f_i},\\
&= q^{-s(2(d-1)n-1)}q^{2(d-1)n} \left(\sum_{f\in A_d^{>0}}q^{-2\deg f} \right)^n.
\end{align*}

\noindent Summing over polynomials of each degree, we obtain that

\[\mathcal{H}_{\delta_n}^s(\mathcal{Z}_0)=q^{-s(2(d-1)n-1)}q^{2(d-1)n}\left(\sum_{m=1}^{d-1}q^m(q-1)q^{-2m} \right)^n.\]

\noindent This last expression equals

\[\mathcal{H}_{\delta_n}^s=q^{-s(2(d-1)n-1)}q^{2(d-1)n}\dfrac{(q^{d-1}-1)^n}{q^{n(d-1)}}.\]

\noindent Then we have that the Hausdorff dimension (by definition) of $\mathcal{Z}_0$ is the value of $s$ such that
\[q^{-2s(d-1)}\cdot q^{2(d-1)}\cdot \dfrac{q^{d-1}-1}{q^{d-1}}=1.\]

\noindent Taking logarithm base $q$ and dividing by $d-1$ yields 
\[s=\dfrac{1}{2}\left(1+\dfrac{\log_q\left(q^{d-1}-1 \right)}{d-1} \right).\]

\noindent Since $q^{d-1}\geq 2$ for $d\geq 2$, this dimension is positive, and in particular $\mathcal{Z}_0
$ is uncountable.
\end{proof}

The remainder part of this section is devoted to study the asymptotic behavior of $\sum_{i=1}^rn_i.$ In particular, we will prove Theorem \ref{main theo 3}. To do so, we use the following result, which connects the ergodic properties of $T_P$ with the Diophantine approximation properties of Schneider convergents.
Let $N$ be a negative integer. A direct application of Birkhoff theorem (see \cite[Theorem 3.2.1]{vianakrerley}) gives the following proposition.

\begin{Prop}\label{proposition primer retorno}
For Haar-almost every point $z\in \hat K$, the following asymptotic frequencies hold.
$$\lim_{k\to \infty}\dfrac{\#\{i\leq k:n_i(x)=N\}}{k}= \begin{cases}  \mu(\pi\mathcal{O}\setminus \pi^{d-1}\mathcal{O})=1-q^{-d}&\text{ if } N=0\\
\mu(\pi^{dN}\mathcal{O}\setminus \pi^{d(N+1)-1}\mathcal{O})=q^{-dN}-q^{d(N+1)} & \text{ if } N<0.\\
\end{cases}
$$
\end{Prop}

In particular, for Haar-almost every $z\in \hat K$, infinitely many exponents $n_i(z)$ are strictly negative. 

\begin{proof}[Proof of Theorem \ref{main theo 3}] By Proposition \ref{proposition primer retorno}, for all sufficiently large $r$ we have
\[-\sum_{i=1}^r n_i\asymp r\sum_{N=1}^\infty N\left (q^{-Nd}-q^{-(N+1)d}\right).\]

\noindent Evaluating the series, 

\[r\sum_{N=1}^\infty N\left (q^{-Nd}-q^{-(N+1)d}\right)=\dfrac{r}{q^d-1}.\]

\noindent We conclude $-\sum_{i=1}^r n_i$ is asymptotic to a positive constant depending only on $q$
and $d$. Therefore

\[-\sum_{i=1}^r n_i \asymp_{q,d} r. \]
    
\end{proof}

\section{Examples}\label{section examples}

\begin{Example}
    Let $F$ be a field with $\mathrm{char}(F)\neq 2$. Let $P\in F[t]$ be a polynomial with $\deg P\geq 2$, and let $f=bt+c$ be a linear polynomial. Define
    \[z=\dfrac{-f+\sqrt{f^2-4}}{2}\in F(\!(t^{-1})\!).\]

    \noindent We first determine the valuation of $z$. Note that 
\begin{equation}\label{definition of z}
z=\dfrac{1}{2}\left(-bt-c +bt\sqrt{\left(1+\dfrac{c}{bt} \right)^2-\dfrac{4}{b^2t^2}}\right),
\end{equation}
and the power series expansion of the square root is
\begin{equation}\label{expansion raiz}
\sqrt{\left(1+\dfrac{c}{bt} \right)^2-\dfrac{4}{b^2t^2}}=1+\dfrac{c}{bt}-\dfrac{2}{b^2t^2}+O(t^{-3}).
\end{equation}

\noindent By Equations \eqref{definition of z} and \eqref{expansion raiz}, we obtain
$z=-\frac{1}{bt}+O(t^{-2}).$ Thus $\nu(z)=1$, and consequently $n_0=0$. Observe that $z$ satisfies the quadratic equation
$x^2+fx-1=0.$
Equivalently,
$\frac{1}{z}-f=z.$
Therefore, $T_P(z)=z,$cwhich implies that 
$f_1(z)=f.$ It follows that the continued fraction expansion of $z$ is $[0,t,0,t,\ldots]$.  
\end{Example}
\section*{Acknowledgements}
We thank Luis Arenas-Carmona for valuable comments and for having read this manuscript.

M.A. was supported by ANID Fondecyt postdoctoral grant 3261344.

C.B. was partially supported by ANID Fondecyt Initiation Grant No. 11260422.


\bibliographystyle{abbrv}
\bibliography{refs.bib}

\end{document}